\documentclass[12pt,reqno]{amsart}
\usepackage{amssymb}

\usepackage{amscd,amsfonts, mathabx}

\newcommand{\RNum}[1]{\uppercase\expandafter{\romannumeral #1\relax}}

\usepackage{graphicx}

\usepackage[T2A]{fontenc}
\usepackage[utf8]{inputenc}
\usepackage[russian,english]{babel}
\input{int-seidel.def}

\usepackage{mdwlist}

\usepackage{hyperref}
\usepackage{blindtext}
\usepackage{scrextend}

\usepackage{graphicx}
\usepackage{caption}
\usepackage{subcaption}
\usepackage{wrapfig}
\usepackage{dirtytalk}
\usepackage{mathtools}
\usepackage{enumerate}

\usepackage{amsmath}

\usepackage{xcolor}
\usepackage{array}

\usepackage{tikz-cd}
\usetikzlibrary{cd}

\usepackage{comment}

\numberwithin{equation}{section}

\myitemmargin 
\def\exd{\text{\normalfont{d}}\,}

\def\aut{\sf{Aut}}

\def\sw{\sf{SW}}
\def\fsw{\sf{FSW}}

\def\mod{\sf{mod}}
\def\spin{\text{spin$^{\cc}$ structure }}
\def\sq{\sf{sq}}

\def\ss{\mib{s}}

\usepackage{titlesec}
\titleformat{\section}[runin]{\bfseries}{\thesection.}{3pt}{}[.]

\newtheorem*{theorem-non}{Theorem}

\theoremstyle{definition}
\newtheorem{remark}{Remark}

\def\UU{{\mib{U}}}

\myitemmargin 
\usepackage{geometry}

\newgeometry{vmargin={25mm}, hmargin={15mm,19mm}}   

\newtheorem{example}{Example}
\newtheorem{thmx}{Theorem}

\tikzset{%
  symbol/.style={
    draw=none,
    every to/.append style={
      edge node={node [sloped, allow upside down, auto=false]{$#1$}}
    },
  },
}

\usepackage[mathcal]{euscript}

\begin{document}

\title[Seidel's theorem via gauge theory]%
{Seidel's theorem via gauge theory}



\author{Gleb Smirnov}

\subjclass[2010]{Primary }

\date{}

\dedicatory{}

\begin{abstract}
A new proof is given that Seidel's generalized 
Dehn twist is not symplectically isotopic to the idenity.
\end{abstract}

\maketitle



\setcounter{section}{0}

\section{Introduction}
It has been shown by Seidel \cite{Sei-1, Sei-2} that 
symplectic 4-manifolds may admit symplectic 
diffeomorphisms which are smoothly isotopic to the 
identity but not symplectically so. He proved the following 
theorem (see \cite[Cor.\,1.6]{Sei-1}).
\begin{thma}[Seidel, \cite{Sei-1}]\label{t:A}
Let $X$ be a complete intersection that is neither $\cp^1 \times \cp^1$ nor 
$\cp^2$. Then there exists a symplectomorphism 
$\varphi \colon X \to X$ that is smoothly, yet not symplectially, 
isotopic to the identity.
\end{thma}
The symplectomorphism $\varphi = [\tau]^2$ in Thm.\,\ref{t:A} 
is the (square of) generalized 
Dehn twist described below. Seidel computed the Floer cohomology 
group $HF^{*}([\tau])$ of $[\tau]$ 
with its module structure over the quantum cohomology ring $QH^{*}(X)$ 
and then showed 
that it differs from that of $HF^{*}([\tau]^{-1})$. As a result, 
he proved that $[\tau]$ is not isotopic to $[\tau]^{-1}$, 
hence not isotopic to the identity. This note aims to give a different proof of this result, though one that so far works only for K3 surfaces (see Thm.\,\ref{t:myA} below). The new proof does not rely on any Floer-theoretic considerations, but instead follows the approach of 
Kronheimer (see \cite{K}) and uses invariants derived from the Seiberg-Witten equations. 
\smallskip%

We define now Seidel's generalized Dehn twist as 
a Picard-Lefschetz monodromy map of a pencil of surfaces. 
To be specific, we restrict ourselves to the case of 
hypersurfaces of $\cp^3$. 
Inside the space $\cp^{N_d}$ of all hypersurfaces in $\cp^3$ of fixed degree $d$, 
there is a codimension-1 subvariety $\Sigma$ which 
parameterizes singular hypersurfaces. Smooth points of $\Sigma$ corresponds to surfaces 
which have a single double-point singularity. 
Pick a smooth point $p \in \Sigma$ and a small complex disk 
$\Delta$ meeting $\Sigma$ at the point $p$, transverse to $\Sigma$. 
Fix a local parameter $t \colon \Delta \to \cc$ such that $t(p) = 0$. Letting $X_t$ denote the hypersurface corresponding to the point 
$t \in \Delta$, we set
\begin{equation}\label{eq:X}
\calx = \left\{ (t,x) \in \Delta \times \cp^3\,|\, x \in X_t \right\}.
\end{equation}
We let $\left\{ X_t \right\}_{t \in \Delta - \left\{ 0 \right\}}$ 
be the family of non-singular projective 
surfaces obtained from $\calx$ by removing 
the singular fiber $X_0$. The Fubini-Study form 
of $\cp^3$ gives rise to a family of K{\"a}hler forms 
$\omega_t \in H^{1,1}(X_t;\rr)$ for each $t \in \Delta - \left\{ 0 \right\}$. 
Moser's trick says that the family of symplectic manifolds 
$\left\{ (X_t, \omega_t) \right\}_{t \in \Delta - \left\{ 0 \right\}}$ 
is locally-trivial, so there is a representation
\[
\pi_1(\Delta - \left\{ 0 \right\}, t_0) \to \pi_0 (\symp(X_{t_0})),
\]
where $t_0 \in \Delta - \left\{ 0 \right\}$ is some fixed base-point. 
The mapping class corresponding to the generator of 
$\pi_1(\Delta - \left\{ 0 \right\}) \cong \zz$ is called Seidel's 
generalized Dehn twist and it is denoted by $[\tau] \in \pi_0 (\symp(X_{t_0}))$. It is a classical fact that $[\tau]$ acts 
as a reflection in $H_2(X_{t_0};\zz)$. Hence,
\[
[\tau]^2 \in \ker \left[ \pi_0(\symp(X_{t_0})) \to \aut(H_2(X_{t_0};\zz)) \right].
\]
In fact, it is also well known (see \cite{Sei-1,K}) that 
\begin{equation}\label{eq:forget}
[\tau]^2 \in \ker \left[ \pi_0(\symp(X_{t_0})) \to \pi_0(\diff(X_{t_0})) \right].
\end{equation}
What we will prove is:
\begin{thma}\label{t:myA}
If $d = 4$, then $[\tau]^2 \in \pi_0(\symp(X_{t_0}))$ is a non-trivial element.
\end{thma}
\begin{remark}
Throughout the paper we work with $\zz_2$ coefficients. However, if one is willing to work with Seiberg-Witten inavriants over $\zz$, 
it is fairly easy to show that $[\tau]^2$ is non only non-trivial but also 
an element of infinite order.
\end{remark}
We now sketch an argument (due to Kronheimer) 
that establishes \eqref{eq:forget}. 
If $\sq \colon \Delta \to \Delta$ is the map $\sq(t) = t^2$, then 
define a complex-analytic fiber space $\calx'$ as the base change
\begin{equation}\label{d:double}
\begin{tikzcd}
\calx' \arrow{d}{} \arrow{r}{}  & \calx \arrow{d}{}\\
\Delta \arrow{r}{\sq} & \Delta,
\end{tikzcd}
\end{equation}
which is equivalent to setting
\begin{equation}\label{eq:x'}
\calx' = \left\{ (t,x) \in \Delta \times \cp^3\,|\, x \in X_{\sq(t)} \right\}.
\end{equation}
The space $\calx'$ is a non-smooth complex 3-fold with a single 
double-point in the fiber over $0$. We proceed by recalling a result of 
Atiyah which allows us to get rid of the double-point in the central fiber.
\begin{thma}[Atiyah, \cite{Atia}]\label{thm:atia}
There exists a complex-analytic family of 
non-singular surfaces $\caly \to \Delta$ and a morphism of families 
$h \colon \caly \to \calx'$,
\begin{equation}\label{d:resolve}
\begin{tikzcd}[row sep = large, column sep = small]
\caly \arrow{d}{h}  & Y_t \arrow[l, hook] \arrow{d}{h_t}\\
\calx'   & X'_t \arrow[l, hook] ,
\end{tikzcd}
\end{equation}
such that for each $t \in \Delta - \left\{ 0 \right\}$, 
$h_t \colon Y_t \to X'_t$ is an isomorphism, whereas 
$h_0 \colon Y_0 \to X_0$ is the minimal resolution. The exceptional 
divisor of $h_0$, which is a smooth rational curve $C \subset Y_0$ of 
self-intersection number $(-2)$, is embedded in $\caly$ as 
a $(-1,-1)$-curve, that is, a curve whose normal bundle is isomorphic to $\scro(-1) \oplus \scro(-1)$. 
\end{thma}
This is the statement of Thm.\,2 in \cite{Atia} except for the assertion about $C$, 
which is explained in \cite[\S\,3]{Atia}. The family morphism 
$h$ is an isomorphism away from the central fibers, so the monodromy of 
$\left\{ Y_t \right\}_{t \in \Delta - \left\{ 0 \right\}}$ 
is equal to that of 
$\left\{ X'_t \right\}_{t \in \Delta - \left\{ 0 \right\}}$. 
However, as the only singular fiber of $\calx'$ has been replaced by 
a smooth surface, the fibers of $\caly$ are all smooth. Hence, 
the monodromy is smoothly isotopic to the identity and \eqref{eq:forget} follows.

\statebf Acknowledgements. 
\ I would like to 
thank Jianfeng Lin for his suggestion to 
consider the winding number 
as a starting point for studying two-dimensional families of K3 surfaces. Thanks also to Sewa Shevchishin for many useful discussions.

\section{Family Seiberg-Witten invariants}
Here, we briefly recall the definition of 
the Seiberg-Witten invariants in both 
classical and family settings. The given 
exposition is extremely brief, meant mainly to fix 
the notations and also to show just how little of 
Seiberg-Witten theory we need to prove Thm.\,\ref{t:myA}. 
We refer the reader to \cite{Nic} for a comprehensive 
introduction to four-dimensional gauge theory 
(see also \cite{wild} for a quick survey), whereas the theory of 
family invariants can be read from the seminal paper 
\cite{LL}.

Let $X$ be a closed oriented {\itshape simply-connected} 4-manifold endowed 
with a Riemannian metric $g$, a self-dual form $\eta$, and a spin$^{\cc}$-structure $\mathfrak{s}$.
Associated to a spin$^{\cc}$-structure are spinor bundles $W^{\pm}$ and a determinant 
line bundle $L = \det\, W^{+}$ over $X$. Define the monopole map 
\[
\mu \colon \Gamma(W^{+}) \times \cala \to 
\Gamma(W^{-}) \times i \,\Gamma( \Lambda_{+}^2 )
\quad \text{by}\quad
\mu(\varphi,A) := ( \cald^{A}\,\varphi, F^{+}_A - \sigma(\varphi) - i \eta ),  
\]
where $\varphi \in \Gamma(W^{+})$ is a self-dual spinor field, 
$A \in \cala$ is a $\UU(1)$-connection on $L$, and $F^{+}_{A}$ is the self-dual part of the curvature 
$F_{A}$. Finally, $\sigma \colon W^{+} \to i \Lambda^{2}_{+}$ stands for the squaring map.

The monopole space (or, the Seiberg-Witten solution space) is, by definition, 
the zero set of $\mu$, and the moduli space of monopoles, denoted by 
$\scrm_{g,\eta}^{\mathfrak{s}}$, is defined to be the quotient of 
$\mu^{-1}(0)$ by the gauge group 
\[
\scrg = \left\{ g \colon X \to S^1 \right\}.
\]
If now $g$ is an element of $\scrg$, the corresponding 
gauge transformation is given by
\[
g \cdot (\varphi, A) := ( e^{-i f} \varphi, A + 2\, i\, \exd f ),\quad 
\text{where $g = e^{i f}$.}
\]
Such $f$ exists because $X$ is simply-connected. 
Gauge transformations preserve the property of 
being a monopole.

The monopole map 
depends on the choices of Riemannian metric, perturbation form, and spin$^{\cc}$-structure. Moreover, a spin$^{\cc}$ structure itself is defined with respect to a metric on $X$. However, it is explained in 
\cite[\S\, 2.2]{R1} that a spin$^{\cc}$ structure 
for one metric can be extended to all of them. Therefore, we can 
consider a family of monopole maps parametrized by $(g,\eta)$.
Let $\Pi$ denote the space of pairs $(g,\eta)$, where $g$ is a metric on $X$ and $\eta$ is a $g$-self-dual form. Note that $\Pi$ is naturally 
a vector bundle over the space of Riemannian metrics $\calr$ on $X$.
Using the parameterized monopole map, one defines the universal 
moduli space as follows
\[
\mathfrak{M}^{\fr{s}} := \bigcup_{(g,\eta) \in \Pi} \scrm_{(g,\eta)}^{\fr{s}}.
\]
Let $\Pi_{ \text{\normalfont{red}} } \subset \Pi$ be the subset 
of pairs $(g,\eta)$ for which $\mu^{-1}(0)$ contains pairs of the form 
$(0,A)$, the so-called reducible monopoles.
Unless $\varphi = 0$, the stabilizer of $(\varphi,A)$ w.r.t.\,$\scrg$ is trivial. However, the stabilizer of $(0,A)$ is $\UU(1)$. 
Therefore, reducible monopoles obstruct $\mathfrak{M}$ to be a manifold.
In case $\varphi = 0$, the equation $\mu(\varphi,A) = 0$ takes the simple form
\begin{equation}\label{eq:red-simple}
F^{+}_{A} - i\eta = 0.
\end{equation}
A solution to \eqref{eq:red-simple} exists iff
\begin{equation}\label{eq:harmonic_parts_0}
\langle F^{+}_{A} \rangle_{g}  =  i \langle \eta \rangle_{g}, 
\end{equation}
where the brackets in both sides denote 
the self-dual harmonic part of the 2-form in question. Recall here that 
the space $\Gamma(\Lambda^2_{+})$ of self-dual 2-forms 
splits as $H_{g} \oplus \im\text{d}^{+}$, where 
$H_{g}$ stands for the space of 
harmonic self-dual forms on $X$, and 
$\im\text{d}^{+}$ stands for the image of 
$\text{d}^{+} \colon \Gamma(\Lambda^1) 
\to \Gamma(\Lambda^2_{+})$. For abbreviation, we will drop the subscript 
and write $\langle\ \rangle$ instead of $\langle\ \rangle_{g}$ when no confusion can 
arise. Since the (self-dual) harmonic part of a closed form 
depends on its cohomology class but not on the specific representative, 
we restate \eqref{eq:harmonic_parts_0} as
\[
\langle \eta + 2 \pi c_1(L) \rangle = 0. 
\]
Now, we set
\[
\Pi^{*} = \left\{ (g,\eta) \in \Pi\,|\, 
\langle \eta + 2 \pi c_1(L) \rangle_g \neq 0 \right\}.
\]
Let us describe the homotopy type of this space. 
Denote by $\calh$ the vector bundle over $\calr$ whose fiber 
over $g \in \calr$ is the space $H_{g}$ of all $g$-self-dual harmonic forms, 
and denote by $\calh^{*}$ the complement, in $\calh$, of 
the section given by 
\[
 -2 \pi\langle c_1(L) \rangle.   
\]
As $\calh$ is a vector bundle of rank $b^{+}(X)$ on a contractible 
space, it follows that $\calh^{*}$ has the homotopy type of $S^{b^{+}(X)-1}$.
Now, observe that the bundle map 
\begin{equation}\label{map:harmonic-pr}
\Pi \to \calh,\quad (g,\eta) \to 
\langle \eta \rangle
\end{equation}
sits in the diagram
\[
\begin{tikzcd}[row sep = small, column sep = large]
\Pi \arrow{r}{\text{\eqref{map:harmonic-pr}}} & \calh\\
                    \Pi^{*} \arrow[u,hook] \arrow{r}{\text{\eqref{map:harmonic-pr}}} & 
                    \calh^{*}\, \arrow[u,hook]
\end{tikzcd}
\]
and that is has contractible fibers. 
Thus $\Pi^{*}$ has the same homotopy type as $\calh^{*}$.

Let us now consider the piece of $\mathfrak{M}$, denoted by $\mathfrak{M}^{*}$, that lies over 
$\Pi^{*}$. A classical result (see \cite[Lem.\,5]{K-M}) is that the projection 
\[
\pr_{\mathfrak{s}} \colon \mathfrak{M}^{*} \to \Pi^{*}
\]
is a proper Fredholm map of index 
\[
d(\mathfrak{s}) = \frac{1}{4}( c_1^2(\mathfrak{s}) - 3 \sigma(X) -2 \chi(X) ),
\]
where $c_1(\mathfrak{s})$, the Chern class of the 
\spin $\fr{s}$, is simply $c_1(L)$. The Sard-Smale theorem then 
asserts that for a generic 
$(g,\eta) \in \Pi^{*}$ the moduli space 
$\scrm_{(g,\eta)}^{\mathfrak{s}} = 
\pr^{-1}_{\mathfrak{s}}(g,\eta)$ is either empty or 
a compact manifold of dimension $d(\fr{s})$. If 
$d(\fr{s}) = 0$, then $\scrm_{(g,\eta)}^{\fr{s}}$ is zero-dimensional, 
and thus consists of finitely-many points. We call
\begin{equation}\label{eq:sw-def}
\sw_{(g,\eta)}(\fr{s}) := 
\# \left\{ \text{points of $\scrm_{(g,\eta)}^{\fr{s}}$} \right\}\,\mod\,2
\end{equation}
the ($\zz_2$-)Seiberg-Witten invariant for the \spin $\fr{s}$ 
w.r.t.\,$(g,\eta)$. If $b_2^{+}(X) > 1$, then $\Pi^{*}$ is connected, and then (again, by the Sard-Smale theorem), for every 
two pairs $(g_1,\eta_1)$ and $(g_2,\eta_2)$ and 
every generic path $(g_t,\eta_t)$ connecting them, the corresponding 
moduli space
\begin{equation}\label{eq:homotop}
\bigcup_t \scrm_{(g_t,\eta_t)}^{\fr{s}}
\end{equation}
is a smooth one-dimensional manifold which draws a cobordism between 
$\scrm_{(g_0,\eta_0)}^{\fr{s}}$ and $\scrm_{(g_1,\eta_1)}^{\fr{s}}$. Hence,
\[
\sw_{(g_0,\eta_0)}(\fr{s}) = \sw_{(g_1,\eta_1)}(\fr{s}).
\]
Another classical fact 
(see e.g. \cite[Prop.\,2.2.22]{Nic}) is that 
there is a charge conjugation involution $\fr{s} \to -\fr{s}$ 
on the set of spin$^\cc$ structures, which changes the sign of 
$c_1(\fr{s})$, and there is a canonical isomorphism between 
\[
\scrm_{(g,\eta)}^{\fr{s}}\quad \text{and}\quad 
\scrm_{(g,-\eta)}^{-\fr{s}}. 
\]
Hence,
\begin{equation}\label{eq:charge}
\sw_{(g,\eta)}(\fr{s}) = \sw_{(g,-\eta)}(-\fr{s}).
\end{equation}
The corresponding $\zz$-valued Seiberg-Witten invarinats 
are also equal to each other, but only up to sign. 
See \cite[Prop.\,2.2.26]{Nic} for the precise statement.
\medskip%

We now recall a version of the family Seiberg-Witten invariants 
that is adequate for our purpose. This simplest version, along 
with other kinds of the family invariants, were systematically studied 
by Li-Liu in \cite{LL}. Let $B$ be a compact smooth manifold 
and suppose we have a fiber bundle $\calx \to B$ whose 
fibers are diffeomorphic to $X$. Pick a \spin $\fr{s}$ on the vertical tangent 
bundle $T_{\calx/B}$ of $\calx$.
For a family of fiberwise metrics $\left\{ g_b \right\}_{b \in B}$ and 
a family of $g_b$-self-dual forms $\left\{ \eta_b \right\}_{b \in B}$, 
we consider the 
parameterized moduli space
\[
\fr{M}_{(g_b,\eta_b)}^{\fr{s}} := \bigcup_{b \in B} \scrm_{(g_b,\eta_b)}^{\fr{s}}. 
\]
Assume that 
\begin{equation}\label{eq:irreducible}
\langle \eta_b + 2 \pi c_1(L) \rangle \neq 0\quad 
\text{for each $b \in B$.}
\end{equation}
Under this assumption and, perhaps after a small perturbation 
of $\left\{ (g_b,\eta_b) \right\}_{b \in B}$, 
the moduli space $\fr{M}_{(g_b,\eta_b)}^\fr{s}$, 
if not empty, is a manifold of dimension
\[
d(\mathfrak{s},B) = 
\frac{1}{4}( c_1^2(\mathfrak{s}_{X}) - 3 \sigma(X) -2 \chi(X) ) + \dim\,B,
\]
where $\fr{s}_{X}$ is the restriction of $\fr{s}$ to any fiber $X$. 
Suppose that $d(\mathfrak{s},B) = 0$. Then, as in \eqref{eq:sw-def}, 
we define the family Seiberg-Witten invariant 
$\fsw_{(g_b,\eta_b)}(\fr{s})$ by counting the points of $\fr{M}_{(g_b,\eta_b)}^\fr{s}$. 
Just like the ordinary Seiberg-Witten invariants 
(see \eqref{eq:homotop}), the family invariants are unchanged under the homotopies of $(g_b,\eta_b)$ that satisfy \eqref{eq:irreducible}, and, just like the ordinary 
invariants (see \eqref{eq:charge}), they share the conjugation symmetry
\begin{equation}\label{eq:charge-family}
\fsw_{(g_b,\eta_b)}(\fr{s}) = \fsw_{(g_b,-\eta_b)}(-\fr{s}).
\end{equation}
Already this version of the family invariants is capable to detect 
non-trivial families of cohomologous symplectic forms (see \cite{K}) 
as well as to distinguish between different connected components of positive scalar curvature metrics (see \cite{R2}).
\smallskip%

Let us now turn to the special case of $b_2^{+}(X) = 3$. 

\section{Unwinding families}\label{sec:unwinding}
As before, let $\calx$ be a smooth fiber bundle over $B$ with fiber $X$. 
From now on, we assume that $B$ is the 2-sphere $S^2$ and $X$ is the K3 surface. 
Pick a family $\left\{ g_b \right\}_{b \in B}$ 
of fiberwise metrics on the fibers of $\calx$. Suppose now that a fiber $X_b$ is 
given a spin$^{\cc}$ structure $\fr{s}_X$. 
Then it is easy to show (see e.g. \cite[Prop.\,2.1]{Nak}) 
that, under the topological 
assumptions that we imposed on $\calx$, there exists a spin$^{\cc}$ on $T_{\calx/B}$ 
whose restriction to $X$ is $\fr{s}_X$.

The second cohomology group of a $K3$ surface is 
a free $\zz$-module of rank $22$ which, when endowed 
with the bilinear form coming from the cup product, becomes 
a unimodular lattice of signature $(3,19)$.

Let us fix (once and for all) 
an abstract lattice $L$ which is isometric to $H^2(X;\zz)$ and 
an isometry $L \to H^2(X_b;\zz)$, where $b \in B$ is some fixed base-point.
The dependence on the choice of base-point 
will be inessential: as $B$ is simply-connected, the groups 
$\left\{ H^2(X_b;\rr) \right\}_{b \in B}$ 
are all canonically isomorphic. We also need to 
introduce the (open) positive cone 
\[
K = \left\{ \kappa \in L\,|\, \kappa^2 > 0 \right\},
\]
which is homotopy-equivalent to $S^2$.

As before, we let $\calh$ denote the bundle 
on $B$ whose fiber over $b \in B$ is the space 
of harmonic $g_b$-self-dual forms. 
Pick a family $\left\{ \eta_b \right\}_{b \in B}$ of 
$g_b$-self-dual forms. Suppose that $(g_b,\eta_b)$ satisfies 
\[
\langle \eta_b \rangle \neq 0\quad 
\text{for each $b \in B$,}
\]
so that the correspondence 
$b \to \langle \eta_b \rangle$ yields a nowhere vanishing section of 
$\calh$. Then, associated to this section, there is a map:
\[
B \to K,\quad b \to [\langle \eta_b \rangle] \in L,
\]
where the brackets $[\cdot]$ signify the cohomology 
class of $\langle \eta_b \rangle$. Since both $B$ and $K$ are homotopy $S^2$, this map 
has a degree, called the winding number of the family 
$(g_b,\eta_b)$. 
\begin{lem}
Suppose that the winding number of $(g_b,\eta_b)$ vanishes. Then 
\begin{equation}\label{eq:wind-family}
\fsw_{(g_b,\lambda \eta_b)}(\fr{s}) = \fsw_{(g_b,-\lambda \eta_b)}(\fr{s})
\end{equation}
for $\lambda$ sufficiently large.
\end{lem}
\begin{proof}
By choosing $\lambda$ large enough, we can make 
\begin{equation}\label{eq:cutoff}
\lambda^2\, \min_{b \in B} \int_{X_b} \langle \eta_b \rangle^2 > 4\,\pi^2 \max_{b \in B} 
\int_{X_b} \langle c_1(L) \rangle^2,
\end{equation}
so that both $(g_b,\lambda \eta_b)$ and $(g_b,-\lambda \eta_b)$ satisfies 
\eqref{eq:irreducible} for $\lambda$ large enough. Thus, both 
sides of \eqref{eq:wind-family} are well defined. 
Let us show that there exists a 
homotopy between 
$\left\{ \lambda\, \eta_b \right\}_{b \in B}$ and 
$\left\{ -\lambda\, \eta_b \right\}_{b \in B}$ that 
satisfies \eqref{eq:irreducible}. To begin with, 
we can assume that $\eta_b = \langle \eta_b \rangle$ for each $b \in B$. 
This is because:
\[
\text{if $\eta_{b}$ satisfies \eqref{eq:irreducible}, then so does 
$(1-t)\eta_b + t \langle \eta_b \rangle$.
}
\]
If \eqref{eq:cutoff} holds, then 
the range of both maps
\begin{equation}\label{map:mymaps}
b \to \lambda [ \eta_b ],\quad b \to -\lambda [ \eta_b ]
\end{equation}
lies in the complement of the ball $O \subset K$, 
\[
O = \left\{ \kappa \in K\,|\, \kappa^2 < 4\,\pi^2\,\max_{b \in B} 
\langle c_1(L) \rangle \right\}.
\]
Observe that, for 
every map $f \colon B \to K$, there exists a unique 
section $\tilde{f} \colon B \to \calh$ such that the diagram
\[
\begin{tikzcd}
\calh \arrow{dr}{[\cdot]}  & \\
B \arrow{r}{f} \arrow{u}{\tilde{f}} & K
\end{tikzcd}
\]
commutes. Also, if $\sf{Range}\,f \subset K - O$, then 
$\tilde{f}(b)$ satisfies \eqref{eq:irreducible} for each 
$b \in B$. To conclude the proof, it suffices to show 
that 
the maps \eqref{map:mymaps} are homotopic as maps from $B$ 
to $K - O$. Since $K - O$ is a homotopy $S^2$, 
the maps \eqref{map:mymaps} are homotopic iff their degrees are 
equal to each other. This is the case, as the winding number 
of $(g_b, \pm \lambda\, \eta_b)$ is equal to that of 
$(g_b, \pm \eta_b)$, and the latter is zero. \qed
\end{proof}
\smallskip%

Combining \eqref{eq:wind-family} and \eqref{eq:charge-family}, we obtain 
\begin{equation}\label{eq:lin-symmetry}
\fsw_{(g_b,\lambda \eta_b)}(-\fr{s}) = \fsw_{(g_b,\lambda \eta_b)}(\fr{s})\quad 
\text{for $\lambda$ sufficiently large.}
\end{equation}
Let us consider some examples of families that satisfy \eqref{eq:lin-symmetry}.
\begin{example}[Constant families]
Suppose that $\calx \to B$ is a trivial bundle. 
Then there is a canonical family $(g_b,\eta_b)$, corresponding 
to some constant metric $g_b = g$ and self-dual form $\eta_b = \eta$.
Suppose that $\langle \eta \rangle \neq 0$. Then such a family has vanishing winding number, and hence it satisfies 
\eqref{eq:lin-symmetry}. In fact, in this case, both sides of 
\eqref{eq:lin-symmetry} must vanish.
\end{example}
\begin{example}[Symplectic families]\label{ex:taubes}
Now, let us not assume that 
$\calx \to B$ is trivial, or that the family 
$\left\{ g_b \right\}_{b \in B}$ is any special. But let us keep the assumption
\begin{equation}\label{eq:eta-const}
[\langle \eta_b \rangle] = \const \in K.
\end{equation}
In this case, the quantities of equality 
\eqref{eq:lin-symmetry} may not vanish, but 
the equality itself holds.

Suppose now that the fibers of $\calx \to B$ are furnished with 
a family of symplectic forms $\left\{ \omega_b \right\}_{b \in B}$ of some 
constant cohomology class. 
Pick a family $\left\{ J_b \right\}_{b \in B}$ of $\omega_b$-compatible 
almost-complex structures, so that
\begin{equation}\label{eq:metrics-taubes}
g_b(\cdot, \cdot) := \omega_b(\cdot, J_b \cdot) 
\end{equation}
gives rise to a family of fiberwise metrics. Recall here 
that the space of compatible almost-complex structures is non-empty and contractible (see e.g. \cite[Prop.\,4.1.1]{McD-Sa-2}).
Since
\[
[\omega_b] = \const,
\]
the winding number of $(g_b,\omega_b)$ must vanish.
More generally, we may assume 
that the cohomology class of $\omega_b$ is not constant but 
varies over a small range
\[
[ \omega_b ] \sim \const,
\]
so that the mapping
\[
B \to K,\quad b \to [\omega_b]
\]
is homotopic to a constant map. Then 
the winding number of $(g_b,\omega_b)$ would still have to vanish.
There is a canonical way 
to perturb the Seiberg-Witten equation on symplectic 
4-manifolds. This is by setting
\begin{equation}\label{eq:taubes}
\eta_b = -\rho^2 \omega_b + \text{constant term in $\rho$.}
\end{equation}
As $\rho$ grows, the contribution of the second term 
gets small. Hence, the winding 
number of $(g_b,\eta_b)$ is equal to that of $(g_b,\omega_b)$, 
which is zero. Then \eqref{eq:lin-symmetry} becomes
\begin{equation}\label{eq:extra-rho}
\fsw_{(g_b,\eta_b)}(\fr{s}) = \fsw_{(g_b,\eta_b)}(-\fr{s})\quad 
\text{for $\eta_b$ as in \eqref{eq:taubes} and $\rho$ large.}
\end{equation}
\end{example}

\section{Proof of Theorem\,\ref{t:myA}}\label{sec:resolutions}
Let $\calx' \to \Delta$ be the family 
of quartic $K3$'s given by the complex-analytic fiber space 
\eqref{eq:x'}, and let 
$\caly$ be as in Thm.\,\ref{thm:atia}. From \eqref{eq:x'}, 
we have the mapping $\calx' \to \cp^3$ given by 
$(t,x) \to x$. On the other hand, we also have the resolution map 
$h \colon \caly \to \calx'$ suggested by Thm.\,\ref{thm:atia}.
Let $\Omega$ be the Funini-Study form on $\cp^3$ and let 
$\Omega_{\caly}$ be the pull-back of $\Omega$ under the mapping:
\[
\begin{tikzcd}[row sep = tiny, column sep = large]
\caly \arrow{r}{h} & \calx' \arrow[r] & \cp^3\\
                    & (t,x) \arrow[u,symbol=\in] \arrow{r} & 
                    x\,. \arrow[u,symbol=\in]
\end{tikzcd}
\]
Note that, for each $t \in \Delta - \left\{ 0 \right\}$, 
the restriction of $\Omega_{\caly}$ to $Y_t$ is K{\"a}hler. 
But the restriction of $\Omega_{\caly}$ 
to the central fiber $Y_0$ is degenerate. 
We now will construct a family of K{\"a}hler forms on the fibers of 
$\caly$ by perturbing $\Omega$ in a neighbourhood of $Y_0$. 
To this end, recall that every K3 surface is K{\"a}hler. 
Hence, there exists \textit{some} K{\"a}hler form $\vartheta_0$ on $Y_0$. 
From the theory of complex-analytic 
families (see \cite[Thm.\,15]{K-S}), we recall:
\begin{thma}[Kodaira-Spencer, \cite{K-S}]
Let $\left\{ Y_t \right\}$ 
be a complex-analytic family of non-singular varieties. 
If $Y_{t_0}$ carries 
a K{\"a}hler form, then, any fiber $Y_t$, sufficiently close to $Y_{t_0}$, 
also admits a K{\"a}hler form. Moreover, 
given any K{\"a}hler form on $Y_{t_0}$, we can choose 
a K{\"a}hler form on each $Y_t$, which depends differentiably 
on $t$ and which coincides for $t = t_0$ with the given K{\"a}hler 
form on $Y_{t_0}$.
\end{thma}
Having fixed $\vartheta_0$ on $Y_0$, we construct a family of K{\"a}hler forms 
$\left\{ \vartheta_t \right\}_{t \in U}$ for a sufficiently small 
neighbourhood $U$ of $0$ in $\Delta$. Choose a bump function 
$\chi \colon \Delta \to \rr$ which equals $1$ at the center of $\Delta$ and equals 
$0$ outside of the neighbourhood $U$. Set:
\begin{equation}\label{eq:perturb}
\omega_t = \Omega_{\caly}|_{Y_t} +  \varepsilon\,\chi\,\vartheta_t\quad\text{
for $\varepsilon$ positive arbitrary small.} 
\end{equation}
Both forms in the right-hand side of \eqref{eq:perturb} are of type $(1,1)$. 
Furthermore, 
$\Omega_{\caly}|_{Y_t}$ is positive for each $t \in \Delta - \left\{ 0 \right\}$ and 
semi-positive for $t = 0$, while $\vartheta_t$ is positive for each $t \in U$. 
Thus, for every $Y_t$, the form $\omega_t$ is a positive $(1,1)$-form, hence is K{\"a}hler.  
Also, we have that:
\begin{equation}\label{eq:class-eps}
[\omega_t] = \const + O(\varepsilon).
\end{equation}
Since $\vartheta_t$ is zero 
for each $t \in \Delta - U$, it follows that:
\[
[\omega_t] = \const\quad \text{for each $t \in \Delta-U$.}
\]
We now fix an abstract 
symplectic K3 surface $(Y,\omega)$ together with a symplectomorphism 
between $(Y,\omega)$ and $(Y_{t_0}, \omega_{t_0})$ for 
some base-point $t_0 \in \Delta - U$. 
We \textit{assume} now that the monodromy homomorphism 
\begin{equation}\label{map:assume-zero}
\pi_1(\Delta - U, t_0) \to \symp(Y,\omega)\quad \text{is the zero homomorphism.}
\end{equation}
In other words, we assume that there is 
a family of symplectomorphisms 
\[
f_t \colon (Y_t,\omega_t) \to (Y,\omega)\quad \text{for every $t \in \del \Delta$.}
\] 
Via the clutching construction, the family $\left\{ f_t \right\}_{t \in \del \Delta}$ 
corresponds to the quotient space:
\[
\cals = \caly \cup Y/\sim,\quad \text{where $(t,y) \sim f_t(y)$ for each $t \in \del \Delta$ and $y \in Y_t$,}
\]
which is a fiber bundle over the 2-sphere 
\[
B = \Delta/\del \Delta.
\]
Since, for each $t \in \del \Delta$, 
the mapping $f_t$ is a symplectomorphism, it follows that 
$\cals \to B$ is a bundle of symplectic manifolds. 
This bundle is not Hamiltonian: the symplectic forms on the fibers are 
not cohomologous. However, their cohomology classes must obey \eqref{eq:class-eps} and so differ from 
each other very little. 
This latter property is interesting: it confers an extra symmetry to the family Seiberg-Witten invariants of this bundle; see Exm.\,\ref{ex:taubes} in \S,\ref{sec:unwinding}. On the contrary, 
an independent computation will show 
that this symmetry fails for $\cals$, which contradicts \eqref{map:assume-zero}.
\smallskip%

First we will need to analyze the Seiberg-Witten equations on the family 
$\caly$, concerning which 
we recall:
\begin{enumerate}[(i)]
\item There is a smooth rational
$(-2)$-curve $C \subset Y_0$ which is embedded in $\caly$ as 
a $(-1,-1)$-curve. 
\smallskip%

\item There is closed $(1,1)$-form $\Omega$ on $\caly$ which is degenerate 
along $C$. Hence,
\[
\int_{C} \Omega_{\caly} = 0.
\]
As $\caly$ is trivial as a differentiable 
family, we can find a fiber-preserving diffeomorphism
\begin{equation}\label{d:trivial}
\begin{tikzcd}
\Delta \times Y_0 \arrow{d}{} \arrow{r}{\Phi} & \caly \arrow{d}{}\\
\Delta \arrow{r}{\id}  & \Delta.
\end{tikzcd}
\end{equation}
Letting $C_t$ denote $\Phi(\left\{ t \right\} \times C)$, we have:
\[
\int_{C_t} \omega_t \geq 0\text{ for each $t \in \Delta$}\quad\text{and}\quad
\int_{C_t} \omega_t = 0\text{ for each $t \in \Delta - U$. }
\]
\end{enumerate}
Letting $\left\{ g_t\right\}_{t \in \Delta}$ be the family of K{\"a}hler metrics 
associated to $\left\{ \omega_t\right\}_{t \in \Delta}$, we pick 
a spin$^{\cc}$ structure $\fr{s}_{C}$ on $T_{\caly/\Delta}$ which, when 
restricted to $Y_0$, satisfies:
\begin{equation}\label{eq:my-spin}
c_1(\fr{s}_{C}) = c_1(Y_0)(=0) + 2\,[C],
\end{equation}
where $[C] \in H^2(Y_0;\zz)$ is the class dual to $C$. 
We remark that \eqref{eq:my-spin} specifies $\fr{s}_{C}$ uniquely.
Now, we recall that, for each $t \in \Delta$, 
the Levi-Civit{\'a} connection of the K{\"a}hler metric $g_t$ 
induces a canonical $\UU(1)$-connection $A_t$ 
on $K_{Y_t}^{*} = \det_{\cc} T^{*}_{Y_t}$. Set:
\begin{equation}\label{eq:kahler-eta}
    i\,\eta_t = F_{A}^{+} - i \, \rho^2\,\omega_t.
\end{equation}
In \cite{K}, 
Kronheimer has proved  the following statement:
\begin{thma}[Kronheimer, \cite{K}]
Let $\left\{ Y_t \right\}_{t \in \Delta}$ be a 
family of non-singular surfaces (not not necessarily K3 surfaces) given 
as a complex-analytic fiber bundle $\caly \to \Delta$, and let 
$C \subset Y_0$ be a smooth rational $(-1,-1)$-curve. 
Suppose that, for each $t \in \Delta - \left\{ 0 \right\}$, the fiber $Y_t$ 
has no effective divisors that are 
homologous to $C_t$. 
Then, for the spin$^{\cc}$ structure $\fr{s}_{C}$ above, 
$\eta_t$ as in \eqref{eq:kahler-eta}, and $\rho$ 
large enough, the 
parameterized moduli space 
\[
\fr{M}^{\fr{s}_C}_{(g_t,\eta_t)} = \bigcup_{t \in \Delta} \scrm^{\fr{s}_C}_{(g_t,\eta_t)}
\]
consists of a single point which lies above the fiber $Y_0$.
This moduli space is transverse of the correct dimension.
\end{thma}
To analyze the Seiberg-Witten equations on $\cals$ we will use:
\begin{lem}
Let $\left\{ J_t \right\}_{t \in \Delta}$ be the 
family of complex-structures on the fibers of $\caly$. 
There exists another family of $\omega_t$-compatible 
\textbf{almost}-complex structures 
$\left\{\right.\!\hat{J}_t\!\left.\right\}_{t \in \Delta}$, 
which is homotopic to 
$\left\{ J_t \right\}_{t \in \Delta}$, and which satisfies the following:
\begin{enumerate}[(i)]
\item $J_t = \hat{J}_t$ for each $t \in U$;
\smallskip%
\item $\hat{J}_t = (f_t^{-1})_{*} \circ J \circ (f_t)_{*}$ 
for all $t \in \del\Delta$ and 
some $\omega$-compatible almost-complex structure $J$ on 
$(X,\omega)$. This condition 
implies that $\left\{\right.\!\hat{J}_t\!\left.\right\}_{t \in \Delta}$ 
gives a family of almost-complex structures on the fibers of $\cals$.
\end{enumerate}
\end{lem}
\begin{proof}
Follows from the well-known fact that, for each $t \in \Delta$, the 
space of $\omega_t$-compatible almost-complex structures is contractible. \qed
\end{proof}
\smallskip%

If now $\left\{\right.\!\hat{g}_t\!\left.\right\}_{t \in \Delta}$ be the 
family of Hermitian 
metrics defined as $g_b(\cdot, \cdot) := \omega_b(\cdot, \hat{J}_b \cdot)$, then, 
for each $t \in U$, it satisfies $\hat{g}_t = g_t$. Hence, 
given any family of perturbations $\left\{ \eta_t \right\}_{t \in \Delta}$, we have:
\begin{equation}\label{eq:decomposition}
   \fr{M}^{\fr{s}_{C}}_{(\hat{g}_t,\eta_t)} = 
   \fr{M}^{\fr{s}_{C}}_{(1)} \bigcup \fr{M}^{\fr{s}_{C}}_{(2)},\quad
   \text{
   where $\fr{M}^{\fr{s}_{C}}_{(1)} = 
   \bigcup_{t \in U} \scrm^{\fr{s}_C}_{(g_t,\eta_t)}$ and 
   $\fr{M}^{\fr{s}_{C}}_{(2)} = 
   \bigcup_{t \in \Delta - U} \scrm^{\fr{s}_C}_{(\hat{g}_t,\eta_t)}$.
   }
\end{equation}
From Taubes' theory of Gromov invariants (see \cite{Taub-1,Taub-2}), we recall:
\begin{thma}[Taubes, \cite{Taub-1,Taub-2}]\label{t:taubes-vanishing}
Let $(X,\omega)$ be a closed symplectic $4$-manifold, 
$J$ any $\omega$-compatible almost-complex structure, 
and $g$ the associated 
Hermitian metric. Choose a cohomology class 
$\varepsilon \in H^2(X;\zz)$ such that $\varepsilon \cdot [\omega] \leq 0$ 
and let $\fr{s}_{\varepsilon}$ be the spin$^{\cc}$ structure such that:
\[
c_1(\fr{s}_{\varepsilon}) = c_1(X) + 2\,\varepsilon,
\]
Then there exists a special $\UU(1)$-connection $A$ on $K^{*}_{X}$ such that for  
the family of perturbation 
\begin{equation}\label{eq:taubes-for-symplectic}
i\,\eta = F_A^{+} - i\,\rho^2\,\omega\quad\text{and}\quad\text{$\rho$ large enough,}
\end{equation}
the moduli space $\scrm^{\fr{s}_{\varepsilon}}_{(g,\eta)}$ is empty. 
If $(X,\omega)$ is K{\"a}hler, then the special connection $A$ is 
the one induced by the Levi-Civit{\'a} connection.
\end{thma}
This theorem is also explained in the notes for Ch.\,10 in \cite{wild}.
\smallskip%

Thus, if $\eta_t$ is as in \eqref{eq:taubes-for-symplectic}, then 
$\fr{M}^{\fr{s}_{C}}_{(2)}$ will be empty, whereas $\fr{M}^{\fr{s}_{C}}_{(1)}$ 
will consist of a single point. On the other hand, restricting 
the spin$^{\cc}$ structure $-\fr{s}_{[C]}$ on $Y_t$, we get:
\[
c_1(-\fr{s}_{[C]}) = -2\,[C],
\]
since $c_1(Y_t) = 0$. Combining this with Thm.\,\ref{t:taubes-vanishing}, we see that  
$\fr{M}^{-\fr{s}_{C}}_{(\hat{g}_t,\eta_t)}$ is empty for 
$\rho$ sufficiently large. Thus, 
for the Seiberg-Witten invariant of $\cals \to B$, we see that:
\[
\fsw_{(\hat{g}_b,\eta_b)}(\fr{s}_{[C]}) = 1\quad \text{and}\quad \fsw_{(\hat{g}_b,\eta_b)}(-\fr{s}_{[C]}) = 0\quad 
\text{for $\eta_b$ as in \eqref{eq:taubes}.}
\]
But, by \eqref{eq:extra-rho}, these invariants must be equal. 
This contradiction finishes the proof.

\bibliographystyle{plain}
\bibliography{references}

\begin{thebibliography}{10}

\bibitem{Atia}
M.~Atiyah.
\newblock On analytic surfaces with double points.
\newblock {\em Proceedings of the Royal Society of London. Series A,
  Mathematical and Physical Sciences}, 247(1249):237--244, 1958.

\bibitem{K-S}
K.~Kodaira and D.~Spencer.
\newblock On deformations of complex analytic structures, iii. stability
  theorems for complex structures.
\newblock {\em Annals of Mathematics}, 71(1):43--76, 1960.

\bibitem{K}
P.~Kronheimer.
\newblock Some non-trivial families of symplectic structures.
\newblock 1997.
\newblock \url{http://people.math.harvard.edu/~kronheim/diffsymp.pdf}.

\bibitem{K-M}
P.~Kronheimer and T.~Mrowka.
\newblock The genus of embedded surfaces in the projective plane.
\newblock {\em Math.\,Res.\,Lett.}, 1:797--808, 1994.

\bibitem{LL}
{T.-J.} Li and {A.-K.} Liu.
\newblock Family {S}eiberg-{W}itten invariants and wall-crossing formulas.
\newblock {\em Communications in Analysis and Geometry}, 9(4):777--823, 2001.

\bibitem{McD-Sa-2}
D.~McDuff and D.~Salamon.
\newblock {\em {I}ntroduction to {S}ymplectic {T}opology}.
\newblock Oxford Graduate Texts in Mathematics. Oxford University Press, third
  edition, 2017.

\bibitem{Nak}
N.~Nakamura.
\newblock The {S}eiberg-{W}itten equations for families and diffeomorphisms of
  4-manifolds.
\newblock {\em Asian Journal of Mathematics}, 7(1):133--138, 2003.

\bibitem{Nic}
L.~Nicolaescu.
\newblock {\em {N}otes on {S}eiberg-{W}itten theory}.
\newblock Graduate studies in mathematics. American Mathematical Society,
  Providence, R.I.

\bibitem{R1}
D.~Ruberman.
\newblock An obstruction to smooth isotopy in dimension 4.
\newblock {\em Math. Res. Lett.}, 5(6):743--758, 1998.

\bibitem{R2}
D.~Ruberman.
\newblock Positive scalar curvature, diffeomorphisms and the {S}eiberg-{W}itten
  invariants.
\newblock {\em Geom. Topol.}, 5(2):895--924, 2001.

\bibitem{wild}
A.~Scorpan.
\newblock {\em The Wild World of $4$-manifolds}.
\newblock American Mathematical Society, 2005.

\bibitem{Sei-2}
P.~Seidel.
\newblock Graded lagrangian submanifolds.
\newblock {\em Bulletin de la Soci\'et\'e Math\'ematique de France},
  128(1):103--149, 2000.

\bibitem{Sei-1}
P.~Seidel.
\newblock {\em Lectures on four-dimensional Dehn twists}, pages 231--267.
\newblock Springer Berlin Heidelberg, 2008.

\bibitem{Taub-1}
C.~Taubes.
\newblock The {S}eiberg-{W}itten invariants and symplectic forms.
\newblock {\em Mathematical Research Letters}, 1(6):809--822, 1994.

\bibitem{Taub-2}
C.~Taubes.
\newblock More constraints on symplectic forms from the {S}eiberg-{W}itten
  invariants.
\newblock {\em Math. Res. Lett.}, pages 9--13, 1995.

\end{thebibliography}

\end{document}